\documentclass[a4paper,12pt]{amsart}

\usepackage{amsmath, amsthm, amssymb, graphicx, nicefrac, footnote, comment}
\usepackage[utf8]{inputenc}

\usepackage{url}
\usepackage[small,nohug,heads=vee]{diagrams}
\diagramstyle[labelstyle=\scriptstyle]

\usepackage[table]{xcolor}
\usepackage{xy}
\usepackage{array}
\usepackage{tikz}
\usepackage{tikz-cd}
\usepackage{mathtools}

\usetikzlibrary{decorations.pathmorphing} 
\usetikzlibrary{fit}					
\usetikzlibrary{backgrounds}	

\usepackage{graphics}

\definecolor{cof}{RGB}{219,144,71}
\definecolor{pur}{RGB}{186,146,162}
\definecolor{greeo}{RGB}{91,173,69}
\definecolor{greet}{RGB}{52,111,72}

\newtheorem{TheoM}{Theorem}

\newtheorem{teo}{Theorem}[section]

\newtheorem{pro}[teo]{Proposition}
\newtheorem{cor}[teo]{Corollary}

\theoremstyle{definition}

\theoremstyle{definition} 
\newtheorem{defi}[teo]{Definition}

\newcommand{\nat}{\mathbb{N}}

\newcommand{\inte}{\mathbb{Z}}

\newcommand{\field}{\mathbb{K}}
\newcommand{\Ffield}{\mathbb{K}}

\title[C--M Property and linearity of pinched Veronese Rings]{Cohen-Macaulay Property and linearity of pinched Veronese Rings}
\author{Ornella Greco, Ivan Martino}

\date{\today}

\begin{document}
\begin{abstract}
In this work, we study the Betti numbers of  pinched Veronese rings, by means of the reduced homology of squarefree divisor complexes. 
We characterize when these rings are Cohen-Macaulay and we the study the shape of the Betti tables for the pinched Veronese in the two variables. As a byproduct we obtain information on the linearity of such rings.
Moreover, in the last section we compute the canonical modules of the Veronese modules.
\end{abstract}
\maketitle

\vspace{-0.2cm}
The Veronese embedding injects the projective space $\mathbb{P}^{n-1}$ into $\mathbb{P}^{N-1}$ by sending $\mathbf{x}=[x_1:x_2:\dots :x_n]$ to the point with projective coordinates all possible monomials $x_1^{i_1} \dots x_n^{i_n}$ of degree $d$, so we set $N=\binom{n+d-1}{d}$. 
The $d$-Veronese ring, $S^{(d)}$, is the coordinate ring of the image of the $d$-th Veronese embedding of $\mathbb{P}^{n-1}$, with $S=\field[x_1, \dots ,x_n]$.


The \emph{pinched Veronese map} is another embedding of $\mathbb{P}^{n-1}$ into $\mathbb{P}^{N-2}$, but this time the components of the image of $\mathbf{x}$ are all but one of the possible monomials. We denote such monomial by $\mathbf{x}^{\textbf{m}}$.
The coordinate ring of the latter image of $\mathbb{P}^{n-1}$ is called \emph{pinched Veronese rings}, $P_{n,d,\mathbf{m}}$, and it is the quotient of $S^{(d)}$ by the principal ideal $(\mathbf{x}^{\textbf{m}})$.
 
\vspace{0.1cm}
 
The koszul property of the pinched Veronese rings was a trendy topic in literature. Peeva and Sturmfels asked whether the pinched Veronese ring $P_{3,3,(1,1,1)}$ is Koszul. A positive answer was given by Caviglia in \cite{caviglia}, and then reproved by Caviglia and Conca in \cite{CavigliaConca}, and, after, in \cite{ConcaKoszul}; later, Tancer generalized this result to $P_{n,n, (1,\dots, 1)}$, see \cite{tancer}. In \cite{ThanhVu}, Vu used a combinatorial approach to prove that $P_{n,d,\mathbf{m}}$ is  Koszul, unless $d\geq 3$ and $\mathbf{m}$ is one of the permutations of $(d-2,2,0,\dots ,0)$. 

\vspace{0.1cm}

On a different note, the community has expressed interest in the $p$--linearity of the resolution of the Veronese ring, that is that $\beta_{i}(S^{(d)})=\beta_{i,(i+1)d}(S^{(d)})$, for all $0<i\leq p$. 
In other words, the $p$--linearity of the Veronese ring implies that the first $p$ columns of the Betti table are precisely as in the picture right below:
\begin{center}
\tikzstyle{known}=[circle,
                                    thick,
                                    minimum size=0.05cm,
                                    draw=purple!80,
                                    fill=purple!20]

\tikzstyle{unknown}=[circle,
                                    thick,
                                    minimum size=0.05cm,
                                    draw=blue!80,
                                    fill=blue!40]

\tikzstyle{background}=[rectangle,
                                                fill=gray!10,
                                                inner sep=0.3cm,
                                                rounded corners=5mm]

\begin{tikzpicture}[>=latex,text height=0.1ex,text depth=0.1ex]
  
  \matrix[row sep=0.1cm,column sep=0.2cm] {
		\node(0) {\scriptsize{$0$}}; & \node {\scriptsize{$1$}}; & \node {\scriptsize{$2$}}; &\node{\scriptsize{$\dots$}}; & \node  {\scriptsize{$p$}}; &\node {\scriptsize{$p+1$}}; \\
		\node (1)  [known]{}; & &  &  & \\
		 & \node (2)  [known]{}; & \node (3)  [known]{};& \node{$\dots$}; & \node (4)  [known]{}; &\node (4)  [unknown]{};& \node (4){$\dots$}; \\
		 &  & & &  &\node (4)  [known]{};& \node(4){$\dots$}; \\
	};

    \begin{pgfonlayer}{background}
        \node [background,
                    fit=(0) (4), label=left:$p$-linear resolution] 
                     {};
    \end{pgfonlayer}
\end{tikzpicture}
\end{center}
The resolution is, then, called linear if $\beta_{i}(S^{(d)})=\beta_{i,(i+1)d}(S^{(d)})$, for all $i>0$.
Using the Eagon-Northcott resolution, that the Veronese ring in two variable is always linear resolution; another proof of this fact can be found in \cite{OttavianiPaolettiVeronese, noi}. We are going to see in Theorem \ref{Main-Shape} that this is not the case for the pinched Veronese ring, see Table \ref{Table-shape-Betti}.
With higher number of variables, several authors have given contributions to the study of the linearity of the Veronese, such as \cite{OttavianiPaolettiVeronese, BrunsConcaRomer, rubei, GotoWatanabe, AramovaBarcanescuHerzog, noi}. 

\vspace{0.15cm}

In this work, we study the Betti numbers of pinched Veronese rings, 
via the reduced homology of certain simplicial complexex, by pursuing the approach used in \cite{noi}.
There are a lot of articles that relate the Betti numbers of semigroup rings with the topological properties of some simplicial complexes. Campillo and Mariju\'{a}n \cite{CampilloMarijuan} introduced a simplicial complex whose reduced homology encodes information  about numerical semigroup. Later, in order to calculate the Betti numbers of affine semigroup rings, Bruns and Herzog \cite{BrunsHerzogSemi} reintroduced Campillo and Mariju\'{a}n's simplicial complex, calling it the \emph{squarefree divisor complex}.
In this article, we use the formula by Bruns and Herzog applied to $P_{n,d,\mathbf{m}}$. Similar combinatorial approaches have been used in literature: for instance, in \cite{Spaul}, the author provided a formula relating the Betti numbers of the $d$-th Veronese embedding of weighted projective spaces to the homology of the pile simplicial complex.

\vspace{0.1cm}

For a clear presentation of the new results, let us fix some notations.
Let $\field$ be an arbitrary field. For $n,d \in \mathbb{N}$, consider the set 
\[
\mathcal{A}_{n,d}= \left\{  (a_1,a_2, \dots , a_n) \in \mathbb{N}^n \left| \ \sum_{i=1}^{n}a_i=d \right. \right\}.
\]
Set $N=\#\mathcal{A}_{n,d}=\binom{n+d-1}{d}$ and pick an element $\mathbf{m}=(m_1,\dots, m_n)$ in $\mathcal{A}_{n,d}$. The \emph{pinched Veronese ring},
\[
  P_{n,d,\mathbf{m}}=\field[\mathbf{x}^{\mathbf{a}}:\mathbf{a}\in \mathcal{A}_{n,d} \setminus \{ \mathbf{m}\}],
\]
is the affine semigroup ring generated by all monomial with degree vector in $\mathcal{A}_{n,d} \setminus \{ \mathbf{m} \}$. 
We denote by $P_{n,d,\mathbf{m}}(z)$ the Hilbert series of $P_{n,d,\mathbf{m}}$.

Many of our results show that the algebraic properties of the ring $P_{n,d,\mathbf{m}}$ changes dependently to the choice of $\mathbf{m}$. For this reason, it is handy the short notation $\operatorname{max}\mathbf{m}=\operatorname{max}\{m_1, \dots , m_n\}$. 

\vspace{0.1cm}

Our first result, in Section \ref{Sec:Hilbert-series}, gives an explicit formula for the Hilbert series $P_{n,d,\mathbf{m}}(z)$.

\begin{TheoM}\label{Main-Hilbert-Series}
Let $P_{n,d, \mathbf{m}}$ be the pinched Veronese ring. Then its Hilbert series is
\begin{equation*}
P_{n,d, \mathbf{m}}(z)=\frac{1}{(n-1)!} \frac{d^{n-1}}{dz^{n-1}}\left[\frac{z^{n-1}}{1-z^d}\right]
			      - \frac{z^d}{(1-z^d)^q}
\end{equation*}
where 
\[
q=\begin{cases}
				  \displaystyle n &\mbox{ if } \operatorname{max}\mathbf{m} =d;\\
				  \displaystyle 1 &\mbox{ if }  \operatorname{max}\mathbf{m} =d-1;\\
				  \displaystyle 0 &\mbox{ otherwise.}	  
                                \end{cases}
\]
\end{TheoM}

This computation allows us to obtain the h-polynomial for the pinched Veronese ring in two variables, see equations (\ref{eq-h-poly-n=2-max=d}), (\ref{eq-h-poly-n=2-max=d-1}), and (\ref{eq-h-poly-n=2-max<d-1}). The knowledge of its coefficients is crucial to the proof of Theorem \ref{Main-Shape}.

Next, we characterize the Cohen--Macaulayness of such rings and, again, we take particular interest in the two variables case, see Section \ref{section:nvariables}.

\begin{TheoM}\label{Main-CM}
The pinched Veronese ring $P_{n,d,\mathbf{m}}$ is Cohen--Macaulay if and only if either $\operatorname{max}\mathbf{m}=d$ or $\operatorname{max}\mathbf{m}=d-1$ and $n=2$. 

Moreover, if $\operatorname{max}\mathbf{m}=d-1$ and $n=2$, then $P_{2,d,\mathbf{m}}$ is also Gorenstein.
\end{TheoM}

The latter result, is the key to study, in Section \ref{sec:linearitypinched}, the linearity of the Betti table of $P_{2,d, \mathbf{m}}$.
To simplify the notations, we set $\mathbf{m}_i=(i,d-i)$, $\mathcal{A}_i=\mathcal{A}_{2,d}\setminus\{\mathbf{m}_i\}$ and $P_{d, i}:=P_{2, d, \mathbf{m}_i}$.
Since $P_{d, i}$ is isomorphic to $P_{d, d-i}$, we assume without loss of generality that $i\leq \lceil \nicefrac{d}{2}\rceil$.

\begin{TheoM}\label{Main-Shape}
  The shape of the Betti table of $P_{2,d, \mathbf{m}}$ is in the Table \ref{Table-shape-Betti}: $P_{d, 0}$ is linear, $P_{d, 1}$ is $(d-3)$--linear, and $P_{d, 1<i \leq \lceil \nicefrac{d}{2}\rceil}$ is $(i-2)$--linear.
  In particular, a close formula for the known Betti numbers (pink circle) is given in Theorem \ref{theo-Betti-max=d}, \ref{theo-Betti-max=d-1}, and \ref{theo-Betti-max<d-1}.
\end{TheoM}

\begin{table}[h]
\begin{tabular}{ m{2cm} m{10cm} }
$P_{2,d}$ & 
\tikzstyle{known}=[circle,
                                    thick,
                                    minimum size=0.05cm,
                                    draw=purple!80,
                                    fill=purple!20]

\tikzstyle{unknown}=[circle,
                                    thick,
                                    minimum size=0.05cm,
                                    draw=blue!80,
                                    fill=blue!40]

\tikzstyle{background}=[rectangle,
                                                fill=gray!10,
                                                inner sep=0.3cm,
                                                rounded corners=5mm]

\begin{tikzpicture}[>=latex,text height=0.1ex,text depth=0.1ex]
  
  \matrix[row sep=0.1cm,column sep=0.05cm] {
		\node(0) {\scriptsize{$0$}}; & \node {\scriptsize{$1$}}; & \node {\scriptsize{$2$}}; &\node{\scriptsize{$\dots$}}; & \node  {\scriptsize{$d-2$}};\\
		\node (1)  [known]{}; & &  &  & \\
		 & \node (2)  [known]{}; & \node (3)  [known]{};& \node{$\dots$}; & \node (4)  [known]{};\\
	};

    \begin{pgfonlayer}{background}
        \node [background,
                    fit=(0) (4),] {};
    \end{pgfonlayer}
\end{tikzpicture}\\
$P_{2,d-1}$ & 
\tikzstyle{known}=[circle,
                                    thick,
                                    minimum size=0.05cm,
                                    draw=purple!80,
                                    fill=purple!20]

\tikzstyle{unknown}=[circle,
                                    thick,
                                    minimum size=0.05cm,
                                    draw=blue!80,
                                    fill=blue!40]

\tikzstyle{background}=[rectangle,
                                                fill=gray!10,
                                                inner sep=0.3cm,
                                                rounded corners=5mm]

\begin{tikzpicture}[>=latex,text height=0.1ex,text depth=0.1ex]
  
  \matrix[row sep=0.1cm,column sep=0.05cm] {
		\node(0) {\scriptsize{$0$}}; & \node {\scriptsize{$1$}}; & \node {\scriptsize{$2$}}; &\node{\scriptsize{$\dots$}};& \node  {\scriptsize{$d-3$}}; & \node  {\scriptsize{$d-2$}};\\
		\node (1)  [known]{}; & &  &  & & \\
		 & \node (2)  [known]{}; & \node (3)  [known]{};& \node{$\dots$}; & \node (4)  [known]{}; & \\
		  & &  &  & & \node (5)  [known]{}; \\
	};

    \begin{pgfonlayer}{background}
        \node [background,
                    fit=(0) (5),] {};
    \end{pgfonlayer}
\end{tikzpicture}\\
$P_{2,1<i \leq \lceil \nicefrac{d}{2}\rceil}$ & 
\tikzstyle{known}=[circle,
                                    thick,
                                    minimum size=0.05cm,
                                    draw=purple!80,
                                    fill=purple!20]

\tikzstyle{unknown}=[circle,
                                    thick,
                                    minimum size=0.05cm,
                                    draw=blue!80,
                                    fill=blue!40]

\tikzstyle{background}=[rectangle,
                                                fill=gray!10,
                                                inner sep=0.3cm,
                                                rounded corners=5mm]

\begin{tikzpicture}[>=latex,text height=0.1ex,text depth=0.1ex]
  
  \matrix[row sep=0.1cm,column sep=0.01cm] {
\node(0) {\scriptsize{$0$}}; & \node {\scriptsize{$1$}}; & \node {\scriptsize{$2$}}; &\node{\scriptsize{$\dots$}};
& \node  {\scriptsize{$i-2$}}; & \node  {\scriptsize{$i-1$}}; & \node  {\scriptsize{$i$}}; &\node{\scriptsize{$\dots$}}; & \node  {\scriptsize{$d-4$}}; & \node  {\scriptsize{$d-3$}}; & \node  {\scriptsize{$d-2$}};& \node  {\scriptsize{$d-1$}};\\		
		\node (1)  [known]{}; & &  &  & & \\
		 & \node (2)  [known]{}; & \node (3)  [known]{};& \node{$\dots$}; & \node (4)  [known]{}; & \node (5)  [known]{}; & \node (5)  [unknown]{};&\node{$\dots$};& \node (6)  [unknown]{};& \node (6)  [unknown]{};\\
		  & &  &  & & \node (7)  [unknown]{}; & \node (5)  [unknown]{}; &\node{$\dots$};  & \node (6)  [unknown]{}; & \node (8)  [known]{}; & \node (9)  [known]{}; & \node (10)  [known]{};\\
	};

    \begin{pgfonlayer}{background}
        \node [background,
                    fit=(0) (10),] {};
    \end{pgfonlayer}
\end{tikzpicture}
\end{tabular}

\caption{The shape of the Betti table of $P_{2,d, \mathbf{m}}$. We provide a close formula for the the Betti numbers identified by pink circles, see Section \ref{sec:linearitypinched}.}
\label{Table-shape-Betti}
\end{table}

To conclude, in the last section, we extend results by Goto and Watanabe \cite{GotoWatanabe} and by Bruns and Herzog \cite{BrunsHerzog}. They have shown that the canonical module of $S^{(d)}$ is given by the Veronese module $S_{n,d,d-n}$.

\begin{TheoM}
Let $k<d$, the canonical module of $S_{n,d,k}$ is $S_{n,d,t}$, with  $t\equiv -n-k \ \mathrm{mod}(d).$
\end{TheoM}

\textbf{Acknowledgements}\\
We would like to thank Ralf Fr\"oberg and Mats Boij, for their great support and useful suggestions. 

\setcounter{section}{0}\setcounter{TheoM}{0}
\section{Preliminaries}\label{sec-Preliminaries}
In this section, we recall the concepts of affine semigroups rings and of normal semigroups, the definition of (pinched) Veronese subring and we give an overview of some results that connect the Betti numbers of semigroup rings to the homology of the \emph{squarefree divisor complex}, given in \cite{CampilloMarijuan, BrunsHerzogSemi}.

\subsection{Semigroup rings}
\begin{defi}
  Let $H$ be an affine semigroup contained in $\mathbb{N}^n$, finitely generated by the set $\mathcal{M}=\{\mathbf{m}_1, \dots,\mathbf{m}_t \}$. Let us consider the affine semigroup ring associated to $H$, i.e. $\field[H]=\field[\mathbf{x}^{\mathbf{m}_i}| \ i=1,\dots, t]$. 
\end{defi}

A presentation of $\field[H]$ is given by 
\begin{eqnarray*}
  \phi: \field[y_1,\dots , y_t]&\rightarrow& \field[x_1,\dots , x_n]\\
			    y_i&\mapsto & \mathbf{x}^{\mathbf{m}_i}.
\end{eqnarray*}
Thus, $\field[H]\cong\nicefrac{\field[y_1,\dots , y_t]}{\textrm{ker}(\phi)}$.
An affine semigroup $H$ is called \emph{normal} if it
satisfies the condition: if $mz\in H$  for some $z \in \inte[H]$ and $m \in \nat,m \neq 0$,
then $z\in H$.

\begin{teo}[Hochster, \cite{hochster}]\label{hochster}
If $H$ is a normal semigroup, then $\field[H]$ is a Cohen-Macaulay ring.
\end{teo}

\begin{defi}
  Given an element $\mathbf{h}\in H$, we define the \emph{squarefree divisor complex} to be the simplicial complex  
  \[
    \Delta_{\mathbf{h}}(H)=\left\{F\subseteq \mathcal{M}\left| \ \mathbf{h}-\sum_{\mathbf{a}\in F}\mathbf{a}\in H \right. \right\}.
  \]
\end{defi}

It will be useful to define some handy notations.
Let $F\subseteq \mathcal{M}$. We denote by $\sum F=\sum_{\mathbf{a}\in F}\mathbf{a}\in H$. The simplicial complex $\Delta_{\mathbf{h}}(H)$ is called \emph{squarefree divisor complex} because $F\in \Delta_{\mathbf{h}}(H)$ if and only if $\mathbf{x}^{\sum F}$ divides $\mathbf{x}^{\mathbf{h}}$ in the semigroup ring $\field[H]$.

Given a vector $\mathbf{a}=(a_1, \dots , a_n)\in \nat^n$, we denote by $|\mathbf{a}|$ its \emph{total degree}, i.e. $|\mathbf{a}|=\sum_{i=1}^n a_i$. We often write $|F|=|\sum F|$.
Finally, we define $\operatorname{max}\mathbf{a}=\operatorname{max}\{a_1, \dots , a_n\}$.

The following result was proved by Bruns and Herzog (see \cite[Proposition 1.1]{BrunsHerzogSemi}), and it gives a formula for calculating the Betti numbers of $\field[H]$ by means of the squarefree divisor complex.
\begin{teo}\label{thm-BH}
 Let $i\in \mathbb{Z}$ and $\mathbf{h}\in H$, then
  \[
    \beta_{i,\mathbf{h}}(\field[H])= \operatorname{dim}_\Ffield \tilde{\operatorname{H}}_{i-1}(\Delta_{\mathbf{h}}(H),\Ffield).
  \]
\end{teo}

\begin{defi}
Let $\Delta$ be a simplicial complex with vertex set $V$. Then the \emph{Alexander dual} of $\Delta$ is a simplicial complex defined as follows:
\[
\Delta^*= \left\{ V\setminus F | \ F \notin \Delta \right\}.
\]
\end{defi}

\noindent
The combinatorial Alexander duality Theorem (see \cite[Theorem 1.1]{bjorner}) provides the following isomorphism:
\begin{equation}\label{duality2}
\tilde{\operatorname{H}}_{i-2}(\Delta^*,\Ffield) \cong\tilde{\operatorname{H}}_{\#V-i-1}(\Delta,\Ffield).
\end{equation}

\subsection{(Pinched) Veronese rings} 
For $n,d \in \mathbb{N}$, consider the set 
\[
\mathcal{A}_{n,d}= \left\{  (a_1,a_2, \dots , a_n) \in \mathbb{N}^n\left| \ \sum_{i=1}^{n}a_i=d \right. \right\}.
\]
We set $N=\#\mathcal{A}_{n,d}$. 
The Veronese ring $S^{(d)}$ on $n$ variables is the semigroup ring $\field[\mathbf{x}^\mathbf{a}: \mathbf{a}\in\mathcal{A}_{n,d}]$.
The Veronese module is defined as $S_{n,d,k}=\oplus_{i\ge0}S_{k+id}$, where $S=\field[x_1,\ldots,x_n]$, $n,d,k \in \mathbb{N}$. 
In particular, when $k=0$, $S^{(d)}=S_{n,d,0}$ is the $d$-th Veronese ring in $n$ variables.
A presentation of $S^{(d)}$ is given by 
\begin{eqnarray*}
  \phi:R=\field[y_1,\dots , y_N] &\rightarrow& \field[x_1,\dots , x_n]\\
			    y_i&\mapsto & \mathbf{x}^{\mathbf{a}_i}
\end{eqnarray*}
with $\mathbf{a}_i$ being the $i$-th  element  of $\mathcal{A}$ with respect to the lexicographic order. 
Thus, $S^{(d)}\cong\nicefrac{R}{\textrm{ker}\phi}$. The module $S_{n,d,k}$ is both an $R$-module and a $S^{(d)}$-module.

In \cite{noi}, we have studied some algebraic properties of the Veronese Modules and we have also described the Betti table of the Veronese modules $S_{2,d,k}$.  
For this work, we are going to use only the following result:

\begin{teo}[Theorem 4.1 of \cite{noi}]
  The Veronese ring $S^{(d)}$ has pure resolution and the Betti table is:
  \[
  \begin{array}{c|ccccc}
    &0  &1 &2    &\dots& d-1\\\hline
    0 &1 &0&0&\dots&0\\
 
    1 &0 &\binom{d}{2}&2\binom{d}{3}&\dots&(d-1)\binom{d}{d}
  \end{array}
  \]
  Namely, for $i>0$, $\beta_{i}(S^{(d)})=\beta_{i, (i+1)d}(S^{(d)})=i\binom{d}{i+1}$.
\end{teo}

\begin{defi}
Let $\mathbf{m}\in \mathcal{A}_{n,d}$.
The \emph{pinched Veronese ring} is the affine semigroup ring associated to the semigroup generated by $\mathcal{A}_{n,d} \setminus \{ \mathbf{m} \}$. 
\end{defi}

Since $P_{n,d,\mathbf{m}}$ is an affine semigroup ring, it has a presentation, namely $P_{n,d,\mathbf{m}}\cong\displaystyle \nicefrac{\field[y_1,\dots , y_{N-1}]}{\textrm{ker}(\phi)}$.
By the Auslander-Buchsbaum formula (see \cite{BrunsHerzog}), we have that
$$
\mathrm{pdim}(P_{n,d,\mathbf{m}})+\mathrm{depth}(P_{n,d,\mathbf{m}})=\operatorname{depth}(\field[y_1,\dots , y_{N-1}]),
$$
and $\operatorname{depth}(\field[y_1,\dots , y_{N-1}])=\operatorname{dim}(\field[y_1,\dots , y_{N-1}])=\operatorname{emb}(P_{n,d,\mathbf{m}})=N-1$.
Moreover, one has:
\[
  1\leq \operatorname{depth} (P_{n,d,\mathbf{m}})\leq \operatorname{dim} (P_{n,d,\mathbf{m}})=n
\]
and then
\[
  N-n-1\leq \operatorname{pdim} (P_{n,d,\mathbf{m}})\leq N-2.
\]

It is worth to mention that Hellus, Hoa and St\"uckrad \cite{HellusHoaStuckrad} proved that the Castelnuovo-Mumford regularity of $P_{2,d,\mathbf{m}}$ is $2$, except when $\mathbf{m}=(d, 0) \ \mathrm{or} \ (0,d)$.
Further information about the Castelnuovo-Mumford regularity of semigroup rings can be found for instance in \cite{reg2,reg1}.

\subsection{The Hilbert series of the Veronese rings}
Given finitely generated $\mathbb{N}$-graded $A$-module $T$, where $A$ is a finitely generated $\mathbb{N}$-graded commutative algebra over $\field$. We denote by $T_i$ the homogeneous part of degree $i$.
The Hilbert series of $T$ is the following generating function:
\[
	T(z)=\sum_{i\geq 0} \operatorname{dim}_{\field}(T_i)z^i,
\]
where $\operatorname{dim}_{\field}(T_i)$ is the dimension of $T_i$ as a $\field$-vector space.
We consider $A(z)$ as the Hilbert series of $A$ seen as a module over its self.

\begin{teo}[Theorem 2.1 of \cite{noi}]
$\frac{d}{dz} S_{n,d,k}(z) =n S_{n+1,d,k-1}(z)$.
\end{teo}
\noindent
Hence,
\[
  S_{n,d,k}(z)= \frac{1}{(n-1)!} \frac{d^{n-1}}{dz^{n-1}}\left[ \frac{z^{k+n-1}}{1-z^d}\right].
\]
Let us write down the general formula for $S^{(d)}(z)$:
\begin{equation}\label{eq:HS-2}
  S^{(d)}(z)=\frac{1+(d-1)z^d}{(1-z^d)^2}
\end{equation}




\section{The Hilbert series of pinched Veronese rings}\label{Sec:Hilbert-series}
%
%
In this section, we consider the semigroup $H\subseteq \nat^n$ generated by $\mathcal{A}_{n,d}\setminus \{\mathbf{m} \}$, and for each case we calculate the Hilbert series of the pinched Veronese ring with respect to the standard grading. 

\begin{TheoM}\label{theo-hilbert-series}
Let $P_{n,d, \mathbf{m}}$ be the pinched Veronese ring. Then
\begin{equation*}
P_{n,d, \mathbf{m}}(z)=\frac{1}{(n-1)!} \frac{d^{n-1}}{dz^{n-1}}\left[\frac{z^{n-1}}{1-z^d}\right]
			      - \frac{z^d}{(1-z^d)^q}
\end{equation*}
where 
\[
q=\begin{cases}
				  \displaystyle n &\mbox{ if } \operatorname{max}\mathbf{m} =d;\\
				  \displaystyle 1 &\mbox{ if }  \operatorname{max}\mathbf{m} =d-1;\\
				  \displaystyle 0 &\mbox{ otherwise.}	  
                                \end{cases}
\]
\end{TheoM}
\begin{proof}
There are three cases:

%
Case 1: $\operatorname{max}\mathbf{m}=d$.

We might assume $\mathbf{m}=(d,0,\dots,0)$.
In each degree $td$ there are exactly $\binom{n+t-2}{t-1}$ elements that are not in the semigroup $H$. Namely, they are the $\binom{n+t-2}{t-1}$ largest $n$-tuples with respect to the lexicographic order on $\nat^n$ with $(1,0,\dots, 0)>(0,1,\dots, 0)> \cdots > (0,\dots, 0,1)$.
Applying recursively Pascal's identity, one has
$$
\binom{n+t-2}{t-1}=n+\binom{n}{2}+\binom{n+1}{3}+\cdots + \binom{n+t-3}{t-1}.
$$
The $n$-tuples of degree $td$ missing from $H$ are:
\begin{eqnarray*}
&{}& (td,0, \dots, 0),\\
&{}& (td-1,1,0,\dots, 0), \dots , (td-1,0,0,\dots, 1),\\
&{}& (td-2,2,0,\dots, 0), (td-2,1,1,0\dots, 0), \dots, (td-2,0,\dots ,0, 2),\\
&{}& \vdots \\
&{}& (td-(t-1),t-1,0, \dots, 0),  \dots, (td-(t-1),0,\dots ,0,t-1).
\end{eqnarray*}

Thus, one has that 
\[
  P_{n,d, \mathbf{m}}(z)=S^{(d)}(z)-z^d S^{(1)}(z)
\]
and therefore, by Theorem 2.1 in \cite{noi}, we have that:
\begin{eqnarray*}
  P_{n,d, \mathbf{m}}(z)=\frac{1}{(n-1)!} \frac{d^{n-1}}{dz^{n-1}}\left[\frac{z^{n-1}}{1-z^d}\right]-\frac{z^d}{(1-z^d)^n}.
\end{eqnarray*}

Case 2: $ \operatorname{max}\mathbf{m}=d-1$.

Without loss of generality consider $\mathbf{m}=(d-1,1,\dots,0)$.
Exactly all the  $n$-tuples of the form $(td-1,1,0,\dots,0)$ do not belong to the semigroup  $H$ generated by $\mathcal{A}_{n,d}\setminus \{ \mathbf{m} \}$. Thus:
\[
  P_{n,d, \mathbf{m}}(z)=S^{(d)}(z)-\frac{z^d}{1-z^{d}},
\]
and so, again by  Theorem 2.1 in \cite{noi},
\begin{eqnarray*}
  P_{n,d, \mathbf{m}}(z)=\frac{1}{(n-1)!} \frac{d^{n-1}}{dz^{n-1}}\left[\frac{z^{n-1}}{1-z^d}\right]-\frac{z^d}{1-z^d}
\end{eqnarray*}

Case 3: $ \operatorname{max}\mathbf{m}<d-1$.\\
In this case,  the only element missing from the semigroup $H$ is $\mathbf{m}$. Therefore,
\[
  P_{n,d, \mathbf{m}}(z)=S^{(d)}(z)-z^d
\]
and so 
\begin{eqnarray*}
  P_{n,d, \mathbf{m}}(z)=\frac{1}{(n-1)!} \frac{d^{n-1}}{dz^{n-1}}\left[\frac{z^{n-1}}{1-z^d}\right]-z^d.
\end{eqnarray*}
\end{proof}

\subsection{The h-polinomial in some useful cases}
  
In Section \ref{sec:linearitypinched}, it will be useful to write down the Hilbert series for the pinched Veronese in two variables.
By simply applying Theorem \ref{theo-hilbert-series}, one gets that
\[
	P_{n,d, \mathbf{m}}(z)=\begin{cases}
				  \displaystyle \frac{1+(d-2)z^d}{(1-z^d)^2} &\mbox{ if } \operatorname{max}\mathbf{m} =d;\\
				  \displaystyle \frac{1+(d-2)z^d+z^{2d}}{(1-z^d)^2} &\mbox{ if }  \operatorname{max}\mathbf{m} =d-1;\\
				  \displaystyle \frac{1+(d-2)z^d+2z^{2d}-z^{3d}}{(1-z^d)^2} &\mbox{ otherwise.}	  
                                \end{cases}
\]
We should write $P_{n,d, \mathbf{m}}(z)$ as $\nicefrac{h(z)}{(1-z^d)^d}$, where 
$$h(z)=\sum_{i=0}^{d}\sum_{j\in \inte}(-1)^i \beta_{i,j}(P_{n,d, \mathbf{m}})z^j.$$
Let us denote $h_{\operatorname{max}\mathbf{m}}$ the $h$-polynomial for the ring $P_{n,d, \mathbf{m}}$.
One gets,
{\scriptsize
\begin{eqnarray}
	h_d(z)&=&\sum_{i=0}^{d}(-1)^i\binom{d-1}{i}(i-1)z^i,\label{eq-h-poly-n=2-max=d}\\
	h_{d-1}(z)&=&\sum_{i=0}^{d}(-1)^{i-1}\binom{d}{i}\frac{(i-1)(d-i-1)}{(d-1)}z^i,\label{eq-h-poly-n=2-max=d-1}\\
	h_{<d-1}(z)&=&\sum_{i=0}^{d}(-1)^{i-1}\left[(d-1)\binom{d-2}{i-1}-\binom{d}{i-1}-\binom{d-2}{i}\right]z^i.\label{eq-h-poly-n=2-max<d-1}
\end{eqnarray}}

\section{The Cohen-Macaulay Property}\label{section:nvariables}
In this section, we apply Theorem \ref{thm-BH} to  pinched Veronese rings. Through this method we are able to calculate some Betti numbers and also to determine when the pinched Veronese ring is Cohen-Macaulay.

\begin{TheoM}
The pinched Veronese ring $P_{n,d,\mathbf{m}}$ is Cohen--Macaulay if and only if either $\operatorname{max}\mathbf{m}=d$ or $\operatorname{max}\mathbf{m}=d-1$ and $n=2$. 

Moreover, if $\operatorname{max}\mathbf{m}=d-1$ and $n=2$, then it is also Gorenstein.
\end{TheoM}

We start by dealing with the cases $\operatorname{max}\mathbf{m}=d, d-1$ and, without loss of generality, we assume $\mathbf{m}=(d,0, \dots, 0)$ and, respectively, $\mathbf{m}=(d-1,1,0, \dots, 0)$.
Since we are dealing with affine semigroup rings, we can apply Theorem \ref{hochster} in order to prove the Cohen-Macaulay property of $P_{n,d,(d,0, \dots, 0)}$.

\begin{teo}
If $\operatorname{max}\mathbf{m}=d$, then the pinched Veronese ring $P_{n,d,\mathbf{m}}$ is Cohen--Macaulay.
\end{teo}
\begin{proof}
By Theorem \ref{hochster}, it is enough to prove that the semigroup $H$ generated by $\mathcal{A}_{n,d}\setminus \{ (d,0,\dots ,0)\}$ is normal.\\
Suppose by contradiction that $mz\in H$ for $z\in \inte^n\setminus H$ and $m>0$. Let the total degree of $z$ be $td$, since it does not belong to $H$, $z$ has the form $(td-q, \mathbf{q})$, where $\mathbf{q}$ is a vector in $\nat^{n-1}$ of total degree $q$ and $q\leq t-1$. But $mz=(mtd-mq,m\mathbf{q})$ does not belong to $H$ because $mq\leq mt-m \leq mt-1$.
\end{proof}

\begin{teo}\label{theo-max=d-1-CM}
If $\operatorname{max}\mathbf{m}=d-1$ and $n>2$, then the pinched Veronese ring $P_{n,d,\mathbf{m}}$ is not Cohen--Macaulay.
\end{teo}
\begin{proof}
It is enough to prove that $\beta_{N-n,(N-n+2)d}(P_{n, d, \mathbf{m}})$ is non zero.
Let us recall that the only element of total degree $td$ that is not the semigroup $H$ associated to $\mathcal{A}_{n,d}\setminus \{ (d-1,1,0, \dots, 0)\}$ is $(td-1,1,0,\dots, 0)$.

Let us choose $\mathbf{a}_1, \dots , \mathbf{a}_{N-n+1}$ in $\mathcal{A}_{n,d}\setminus \{ (d-1,1,0, \dots, 0), (d,0, \dots, 0)\}$ and set $\mathbf{h}=(d-1,1,0, \dots, 0)+\sum_{i=1}^{N-n+1}\mathbf{a}_i$. The degree of $\mathbf{h}$ is $(N-n+2)d$. Consider $F=\{\mathbf{a}_1, \dots , \mathbf{a}_{N-n+1}\}$. Since $\mathbf{h}- \sum_{i=1}^{N-n+1}\mathbf{a}_i=(d-1,1,0, \dots, 0)\notin H$, $F$ is not a face of $\Delta_{\mathbf{h}}$. On the other hand, from the fact that $(d,0, \dots, 0)\notin F$, $\mathbf{h}- \sum_{j=1,i\neq i}^{N-n+1}\mathbf{a}_j=(d-1,1,0, \dots, 0)+\mathbf{a}_i\in H$ for each $i=1, \dots, N-n+1$,  thus $F\setminus \{ \mathbf{a}_i\}\in \Delta_\mathbf{h}$. Therefore $\Delta_\mathbf{h}$ has non-zero $(N-n-1)$-homology and, using Theorem \ref{thm-BH}, this implies $\beta_{N-n,(N-n+2)d}(P_{n, d, \mathbf{m}})\neq 0$.
\end{proof}

We are going to treat the pinched Veronese in two variables in the next subsection, because it is also Gorenstein.
For now, we deal with the Cohen-Macaulay property when $\operatorname{max}\mathbf{m}<d-1$.
\begin{teo}
  If $\operatorname{max}\mathbf{m}<d-1$, the projective dimension of $P_{n, d, \mathbf{m}}$ is $N-2$ and $\beta_{N-2,Nd}(P_{n, d, \mathbf{m}})\geq 1$.  
  Thus it is not Cohen-Macaulay.
\end{teo}
\begin{proof}
  It is sufficient to show that $\beta_{N-2,\mathbf{t}}(P_{n, d, \mathbf{m}})=1$ for $\mathbf{t}=\sum_{\mathbf{a}\in\mathcal{A}_{n,d}}\mathbf{a}$.  
  There are $N-1$ vertices, all the elements in $\mathcal{A}_{n,d}\setminus \{\mathbf{m}\}$.
 The only $(N-2)$--dimensional set  $\mathcal{A}_{n,d}\setminus \{\mathbf{m}\}$ is not a face of $\Delta_{\mathbf{t}}$ since  $\mathbf{t}-\sum_{\mathbf{a}\in \mathcal{A}_{n,d}\setminus \{\mathbf{m}\}}\mathbf{a}=\mathbf{m}$, that is not in the semigroup.
  Instead, all the $(N-3)$--dimensional subsets of $\mathcal{A}_{n,d}\setminus \{\mathbf{m}\}$ belong to $\Delta_{\mathbf{t}}$, because all the  elements of total degree $2d$ are in  the semigroup.
  %
  So, in this case $\Delta_\mathbf{t}$ is the boundary of a $(N-2)$-dimensional simplex. 
  Using Theorem \ref{thm-BH}, we conclude that $\beta_{N-2, \mathbf{t}}=\operatorname{dim}_\Ffield \tilde{\operatorname{H}}_{N-3}(\Delta_\mathbf{t})=1$.
\end{proof}

\subsection{The Veronese pinched in $(d-1,1)$ or $(1,d-1)$ is Gorestein}
When we reduce to two variables, $n=2$, the proof of Theorem \ref{theo-max=d-1-CM} does not work because $N=d+1$ and $(N-n+2)=d+1$, so we cannot  obtain $\mathbf{h}$ as a sum of $d+1$ elements avoiding $(d-1,1)$. 
Indeed, in this case the pinched Veronese ring $P_{2,d,(d-1,1)}$ is not just Cohen-Macaulay, but also Gorenstein.
We start by proving the Cohen-Macaulayness with the same tools of the proof of Theorem \ref{theo-max=d-1-CM}.

\begin{pro}
If $\operatorname{max}\mathbf{m}=d-1$ and $n=2$, then the pinched Veronese ring $P_{n,d,\mathbf{m}}$ is Cohen--Macaulay.
\end{pro}
\begin{proof}
Let us show that $\beta_{d-1}(P_{2,d,(d-1,1)})=0$. 
In this case, Hellus, Hoa and St\"uckrad \cite{HellusHoaStuckrad} proved that the Castelnuovo-Mumford regularity of $P_{2,d,\mathbf{m}}$ is $2$. Thus, it is enough to show that $\beta_{d-1,d^2}(P_{2,d,(d-1,1)})=\beta_{d-1,(d+1)d}(P_{2,d,(d-1,1)})=0$.

By Theorem \ref{thm-BH}, $\beta_{d-1,d^2}(P_{2,d,(d-1,1)})=\sum_{|\mathbf{h}|=d^2}\operatorname{dim}_\Ffield \tilde{\operatorname{H}}_{d-2}(\Delta_{\mathbf{h}},\Ffield)$. 
By Alexander duality, see (\ref{duality2}), if $V$ is the set of vertices of $\Delta_{\mathbf{h}}$, then
$$
\tilde{\operatorname{H}}_{d-2}(\Delta_{\mathbf{h}},\Ffield)\cong \tilde{\operatorname{H}}_{\#V-d-1}(\Delta_{\mathbf{h}}^*,\Ffield).
$$
For $\#V<d$, the right hand side of the isomorphism becomes $\tilde{\operatorname{H}}_{j}(\Delta_{\mathbf{h}}^*,\Ffield)$ for $j\leq -2$ and so it is trivial.

We are left to deal with $\#V= d$, i.e. $V=\mathcal{A}_{2,d}\setminus \{ (d-1,1)\}$. If $V\in \Delta_\mathbf{h}$, then $\Delta_\mathbf{h}$ is a simplex, in particular the $(d-2)$-reduced homology is zero. Otherwise, we have that $\mathbf{h}-\sum_{\mathbf{a}\in V} \mathbf{a}\neq (0,0)$, i.e. one of the coordinates of this vector is negative, say the first coordinate, then $\mathbf{h}-\sum_{\mathbf{a}\in V\setminus \{ (0,d)\}} \mathbf{a}$ still has first coordinate negative, that is $V\setminus \{ (0,d)\}=\mathcal{A}_{2,d}\setminus \{ (d-1,1),(0,d)\}\notin \Delta_{\mathbf{h}}$ which implies that $(0,d)\in \Delta_\mathbf{h}^*$ and $\tilde{\operatorname{H}}_{-1}(\Delta_{\mathbf{h}}^*,\Ffield)\cong 0$.

Similarly, $\beta_{d-1,(d+1)d}(P_{2,d,(d-1,1)})=\sum_{|\mathbf{h}|=(d+1)d}\operatorname{dim}_\Ffield \tilde{\operatorname{H}}_{d-2}(\Delta_{\mathbf{h}},\Ffield)$. 
We need to deal only with the case in which the set of vertices, $V$, of $\Delta_{\mathbf{h}}$ coincides with $\mathcal{A}_{2,d}\setminus \{ (d-1,1)\}$. When $\mathcal{A}_{2,d}\setminus \{ (d-1,1)\}\in \Delta_\mathbf{h}$, $\Delta_\mathbf{h}$ is a simplex. Otherwise, $\mathbf{h}-\sum_{\mathbf{a}\in V}\mathbf{a}$ does not belong to the semigroup $H$: either, $\mathbf{h}-\sum_{\mathbf{a}\in V}\mathbf{a}=(d-1,1)$ or $\mathbf{h}-\sum_{\mathbf{a}\in V}\mathbf{a}$ has a negative coordinate; in both cases, there is a vertex $\mathbf{v}\in V=\mathcal{A}_{2,d}\setminus \{ (d-1,1)\}$ such that  $\mathcal{A}_{2,d}\setminus \{ (d-1,1),\mathbf{v}\}\notin \Delta_{\mathbf{h}}$, i.e. such that $\mathbf{v}\in \Delta_\mathbf{h}^*$.
\end{proof}

Now we will prove that it is Gorenstein by looking at the socle of an artinian reduction.

\begin{teo}\label{theo-n=2-gorestein}
  Assume that $\mathbf{m}=(d-1,1)$, then $P_{2,d, \mathbf{m}}$ is Gorenstein.
\end{teo}
\begin{proof}
  Let us consider the presentation map of the pinched Veronese ring:
  \begin{eqnarray*}
  \phi:\field[y_0, \dots , \widehat{y_{d-1}}, y_{d}] &\rightarrow& \field[x_1,x_2]\\
				y_i&\mapsto & x_1^{i}x_2^{d-i}.
  \end{eqnarray*}
  We observe that $\{y_1, y_d\}$ constitutes a regular sequence in $P_{2,d,\mathbf{m}}$ and, so, $T=\nicefrac{P_{2,d,\mathbf{m}}}{(y_0, y_{d})}$ is artinian.
  
  The monomial $x_1^{d+1}x_2^{d-1}$ is the only two monomial of degree $2d$ in $\field[H]$ that is not zero after the quotient by $(y_0, y_{d})$, because the only way to generate it is by multiplying $\phi(y_d)\phi(y_{d-1})$ and $y_{d-1}$ is not a generator.
  Observe that a similar line of argument holds for $x_1^{2d-1}x_2$, but this monomial is not in $\field[H]$ (see Case 2 in Section \ref{Sec:Hilbert-series}). 
  
  Moreover, all the product $y_iy_jy_k$ are zero in $T$, because all can be written as the product of $y_0^{2}$, $y_d^{2}$ or $y_0y_n$, except the ones that are mapped to $x_1^{2d-1}x_2^{d+1}$ and to $x_1^{d-1}x_2^{2d+1}$.
  The latter is the image of $\phi(y_dy_{1}y_{d-2})=x_1^{2d-1}x_2^{d+1}$ and similarly for the other one.
  
  Therefore the binomials $y_iy_j$ so that $\phi(y_iy_j)=x_1^{d+1}x_2^{d-1}$ belong to a class in $T$ and this is also the only generator (of degree $2$).
\end{proof}

\section{The Linearity of the pinched Veronese in two variable}\label{sec:linearitypinched}

In this section, we focus on the shape of the Betti table of the pinched Veronese ring in two variables and, as a byproduct, we obtain information on the linearity of the pinched Veronese ring.

One could note that if $\operatorname{max}\mathbf{m}=d$, then $P_{2,d,\mathbf{m}}$ is isomorphic to the Veronese ring of degree $S^{(d-1)}$. The Betti table of such ring has been well studied and the authors have also computed it in Theorem 4.1 of \cite{noi}. 

\begin{teo}\label{theo-Betti-max=d}
	If $\operatorname{max}\mathbf{m}=d$, the Betti table $P_{2,d, \mathbf{m}}$ is 
\[	\begin{array}{c|ccccc}
    &0  &1 &2    &\dots& d-2\\\hline
    0 &1 &0&0&\dots&0\\
 
    1 &0 &\binom{d-1}{2}&2\binom{d-1}{3}&\dots&(d-2)\binom{d-1}{d-1}
  \end{array}
\]
\end{teo}

Observe that these Betti numbers equal - up to sign - the coefficients of the h-polynomial shown in equation (\ref{eq-h-poly-n=2-max=d}).

In the previous section, we have shown in Theorem \ref{theo-n=2-gorestein} that if $\operatorname{max}\mathbf{m}=d-1$, then $P_{2,d,\mathbf{m}}$ is Gorenstein. Moreover Hellus, Hoa and St\"uckrad \cite{HellusHoaStuckrad} proved that the Castelnuovo-Mumford regularity of $P_{2,d,\mathbf{m}}$ is two, except when $\operatorname{max}\mathbf{m}=d$.
With the computation of the h-polynomial in equation (\ref{eq-h-poly-n=2-max=d-1}) we compute the Betti table of the pinched Veronese also in this case.

\begin{teo}\label{theo-Betti-max=d-1}
	If $\operatorname{max}\mathbf{m}=d-1$, the Betti table $P_{2,d, \mathbf{m}}$ is 
\[	\begin{array}{c|cccccccc}
       &0  &1 &2    &\dots& i &\dots& d-3 &d-2\\\hline
    0 &1 &0&0&\dots&0&\dots&0&0\\
    1 &0 & \binom{d}{2}\frac{d-3}{d-1} & \binom{d}{3}\frac{2(d-4)}{d-1} & \dots & \binom{d}{i+1}\frac{i(d-i-2)}{d-1} &  \dots & \binom{d}{d-2}\frac{d-3}{d-1}&0\\    
    2 &0 &0&0&\dots &0&\dots &0&1
  \end{array}
\]
\end{teo}

Let us finally address the last case, that is $\operatorname{max}\mathbf{m}<d-1$. 
For sake of notation, we set $\mathbf{m}_i=(i,d-i)$, $\mathcal{A}_i=\mathcal{A}_{2,d}\setminus\{\mathbf{m}_i\}$ and we denote the pinched Veronese ring in two variable by $P_{d, i}:=P_{2, d, \mathbf{m}_i}$. 
Along all the rest of the section, we are going to assume for simplicity that $i\leq \lceil d/2\rceil$.

\begin{pro}\label{pro-ultimo-zero}
If $\operatorname{max}\mathbf{m}<d-1$, then $\beta_{d-2,(d-1)d}(P_{d, i})=0$. 
\end{pro}
\begin{proof}
If $\operatorname{max}\mathbf{m}<d-1$, then $\mathbf{m}\neq (0,d), (1,d-1), (d-1,1), (d,0)$ and so the semigroup generated by $\mathcal{A}_i$ is the same semigroup generated by $\mathcal{A}_{2,d}$ with only $\mathbf{m}_i$ missing.
In other words, every monomial, but $\mathbf{x}^{\mathbf{m}_i}$ in the Veronese ring belongs to $P_{d, i}$ and viceversa.
By using \cite{BrunsHerzogSemi}, we know that
\[
\beta_{d-2,(d-1)d}(P_{d, i})=\sum_{|\mathbf{h}|=(d-1)d}\operatorname{dim}_\Ffield \tilde{\operatorname{H}}_{d-3}(\Delta_{\mathbf{h}},\Ffield).
\]
The simplicial complex $\Delta_{\mathbf{h}}$ is a simplicial complex over $\mathcal{A}_i$ and without loss of generality we can assume that every vertex belongs to $\Delta_{\mathbf{h}}$.
Thus by using Alexander duality $\tilde{\operatorname{H}}_{d-3}(\Delta_{\mathbf{h}},\Ffield)\simeq \tilde{\operatorname{H}}_{0}(\Delta_{\mathbf{h}}^*,\Ffield)$. 
We are going to show that $\tilde{\operatorname{H}}_{0}(\Delta_{\mathbf{h}}^*,\Ffield)=0$, by showing that $\Delta_{\mathbf{h}}^*$ is connected.

Let $\sum \mathcal{A}_i\neq \mathbf{h}+\mathbf{m}_j$ for any $j$.
Then $\{\mathbf{m}_0,\mathbf{m}_l\}$ belongs to $\Delta_{\mathbf{h}}^*$, and so $\Delta_{\mathbf{h}}^*$ is connected.

If $\sum \mathcal{A}_i= \mathbf{h}+\mathbf{m}_j$ for some $j< i$, then $\mathbf{m}_j$ is the only vertex that does not belong to $\Delta_{\mathbf{h}}^*$. 
It is easy to observe that $\{\mathbf{m}_0,\mathbf{m}_l\}\in \Delta_{\mathbf{h}}^*$ for any $0\neq l<j$ and $\{\mathbf{m}_d,\mathbf{m}_l\}\in \Delta_{\mathbf{h}}^*$ for any $d\neq l>j$. 
The simplicial complex $\Delta_{\mathbf{h}}^*$ is connected, because $\{\mathbf{m}_k,\mathbf{m}_d\}\in \Delta_{\mathbf{h}}^*$ with $k=j+(i-d)<l$.

Similar argument works if $\sum \mathcal{A}_i= \mathbf{h}+\mathbf{m}_j$ for some $j \geq i$, with the only remark that if $\sum \mathcal{A}_i= \mathbf{h}+\mathbf{m}_i$, then all the verticies in $\mathcal{A}_i$ belong to $\Delta_{\mathbf{h}}^*$.
\end{proof}

\begin{pro}\label{pro-necessaria}
If $\operatorname{max}\mathbf{m}<d-1$, then $\beta_{d-3,(d-2)d}(P_{d, i})\neq 0$. 
\end{pro}
\begin{proof}
We know \cite{BrunsHerzogSemi} that
\[
\beta_{d-3,(d-2)d}(P_{d, i})=\sum_{|\mathbf{h}|=(d-2)d}\operatorname{dim}_\Ffield \tilde{\operatorname{H}}_{d-4}(\Delta_{\mathbf{h}},\Ffield).
\]

Pick $\mathbf{h}$ such that $\Delta_{\mathbf{h}}$ has all the maximal number of verticies, that is $d$.
In this case, by using Alexander duality $\tilde{\operatorname{H}}_{d-4}(\Delta_{\mathbf{h}},\Ffield)\simeq \tilde{\operatorname{H}}_{1}(\Delta_{\mathbf{h}}^*,\Ffield)$. 
Since the total degree $|\mathbf{h}|$ of $\mathbf{h}$ is $(d-2)d$, then every $d-1$-cardinality subset of $\mathcal{A}_i$ does not belongs to $\Delta_{\mathbf{h}}$ and, therefore, its complement, a singleton, belongs to $\Delta_{\mathbf{h}}^*$. So, $\Delta_{\mathbf{h}}^*$ is a simplicial complex over $\mathcal{A}_i$.

We recall that 
\[
	\{\mathbf{m}_l, \mathbf{m}_j, \mathbf{m}_k\}\in\Delta_{\mathbf{h}}^* \Leftrightarrow \mathcal{A}_i\setminus \{\mathbf{m}_l, \mathbf{m}_j, \mathbf{m}_k\}\notin \Delta_{\mathbf{h}}.
\]	
The latter is a subset of cardinality $d-3$ and so $\mathbf{h}-\sum (\mathcal{A}_i\setminus \{\mathbf{m}_l, \mathbf{m}_j, \mathbf{m}_k\})$ has total degree $d$.
Assume that $\mathbf{h}-\sum (\mathcal{A}_i\setminus \{\mathbf{m}_l, \mathbf{m}_j, \mathbf{m}_k\})=\mathbf{m}_0 \neq \mathbf{m}_i$.
Then, by construction, $\{\mathbf{m}_l, \mathbf{m}_j\}, \{\mathbf{m}_l, \mathbf{m}_k\}$ and $\{\mathbf{m}_j, \mathbf{m}_k\}$ are in $\Delta_{\mathbf{h}}^*$, but $\{\mathbf{m}_l, \mathbf{m}_j, \mathbf{m}_k\}$ does not.
This shows that for such $\mathbf{h}$, $\tilde{\operatorname{H}}_{d-4}(\Delta_{\mathbf{h}},\Ffield)$ is non trivial.
\end{proof}

\begin{teo}\label{theo-upto-linearity}
If $\operatorname{max}\mathbf{m}<d-1$, then $P_{d, i}$ is at most $(i-2)$-linear.
More precisely, $\beta_{i-1,(i+1)d}(P_{d, i})\neq 0$. 
\end{teo}
\begin{proof}
Let us apply again the result in \cite{BrunsHerzogSemi} and one knows that
\[
\beta_{i-1,(i+1)d}(P_{d, i})=\sum_{|\mathbf{h}|=(i+1)d}\operatorname{dim}_\Ffield \tilde{\operatorname{H}}_{i-2}(\Delta_{\mathbf{h}},\Ffield).
\]
and we are going to show the statement by showing that for certain $\mathbf{h}$ the homology $\tilde{\operatorname{H}}_{i-2}(\Delta_{\mathbf{h}},\Ffield)$ is non trivial.
Assume $\mathbf{h}=\mathbf{m}_0+\mathbf{m}_1+ \dots +\mathbf{m}_{i}$ and call $F=\{\mathbf{m}_0+\mathbf{m}_1, \dots, \mathbf{m}_{i-1}\}$.
Then, $F$ does not belong to $\Delta_{\mathbf{h}}$, because $\mathbf{h}-\sum F=\mathbf{m}_{i}$ is not in the semigroup $H_i$.
Every subset of $F$ belongs to $\Delta_{\mathbf{h}}$, because $\mathbf{h}-\sum (F\setminus \{\mathbf{m}_{j}\})=\mathbf{m}_{i}+\mathbf{m}_{i}\in H_i$.
\end{proof}

We are going to prove, now, that the pinched Veronese is $(i-2)$-linear. For this, the only thing remains to prove is that $\beta_{i-2,id}(P_{d, i})= 0$.
Call $\Delta^v_\mathbf{h}$ the squarefree divisor complex of degree $\mathbf{h}$ associated to the Veronese embedding $S^{(d)}$
and call $\Delta^{(i)}_\mathbf{h}$ the one associated to the pinched Veronese ring $P_{d, i}$; we are going to show that they are closely related. Remark that 
\[
	\operatorname{Link}(\Delta^v_\mathbf{h}, \mathbf{m}_i):=\{F\in \Delta^v_\mathbf{h}, \exists G\in \Delta^v_\mathbf{h}: \mathbf{m}_i\in G
	\supseteq F\},
\]
and call, shortly, $\operatorname{L}^i_\mathbf{h}=\operatorname{Link}(\Delta^v_\mathbf{h}, \mathbf{m}_i)$.

\begin{pro}\label{pro-decomposition}
$\Delta^v_\mathbf{h}=\Delta^{(i)}_\mathbf{h}\cup \operatorname{L}^i_\mathbf{h}$.
\end{pro}
\begin{proof}
	The simplicial complexes, $\Delta^{(i)}_\mathbf{h}$ and $\operatorname{L}^i_\mathbf{h}$, are naturally subcomplexes of $\Delta^v_\mathbf{h}$. So $\Delta^v_\mathbf{h}\supseteq\Delta^{(i)}_\mathbf{h}\cup \operatorname{L}^i_\mathbf{h}$. 
	
	Let us consider the reverse inclusion and pick $F\in \Delta^v_\mathbf{h}$. We want to show that it $F\notin \Delta^{(i)}_\mathbf{h}$, then $F\in\operatorname{L}^i_\mathbf{h}$.
	Indeed, $F\in \Delta^v_\mathbf{h}\setminus \Delta^{(i)}_\mathbf{h}$ if and only if $\mathbf{m}_i\in F$ or $\mathbf{h}-\sum F=\mathbf{m}_i$. If $\mathbf{m}_i\in F$, then $F\in \operatorname{L}^i_\mathbf{h}$. Now, assume $\mathbf{m}_i\notin F$ and $\mathbf{h}-\sum F=\mathbf{m}_i$, then $F\cup\{\mathbf{m}_i\}$ is a subset of  $\Delta^v_\mathbf{h}$ because $\mathbf{h}-\sum (F\cup\{\mathbf{m}_i\})=\mathbf{h}-\sum F- \mathbf{m}_i=\mathbf{0}$. Hence, 
	$\mathbf{m}_i\in F\cup\{\mathbf{m}_i\}	\supseteq F$ and $F\in\operatorname{L}^i_\mathbf{h}$.
\end{proof}

\begin{pro}\label{pro-dim-intersection}
If $|\mathbf{h}|=id$, then $\operatorname{dim}\left( \Delta^{(i)}_\mathbf{h}\cap \operatorname{L}^i_\mathbf{h}\right) < i-2$.
\end{pro}
\begin{proof}
Indeed, assume $F\in \Delta^{(i)}_\mathbf{h}\cap \operatorname{L}^i_\mathbf{h}$ and assume that the cardinality of $F$ is $i-1$, then $|F|=(i-1)d$. Now, $F\in \Delta^{(i)}_\mathbf{h}$ if and only if $\mathbf{m}_i\notin F$ and $\mathbf{h}-\sum F\in H\setminus \{\mathbf{m}_i\}$.
Hence $\mathbf{h}-\sum F$ is an element of total degree $d$ and it is not supposed to be $\{\mathbf{m}_i\}$. On the contrary, such $F$ also belongs to $\operatorname{L}^i_\mathbf{h}$, hence it is contained in a set $G$ containing $\mathbf{m}_i$. For degree constrains such $G$ is precisely $F\cup \{\mathbf{m}_i\}$ and $\mathbf{h}-\sum F-\mathbf{m}_i=\mathbf{0}$, so a contradiction. 
\end{proof}

By Theorem 4.1 of \cite{noi}, we know that if $|\mathbf{h}|\neq jd$, then $\tilde{\operatorname{H}}_{j-2}(\Delta^v_{\mathbf{h}},\Ffield)=0$. For what may concern next theorem we are going to use that if $|\mathbf{h}|=id$, then 
$\tilde{\operatorname{H}}_{i-1}(\Delta^v_{\mathbf{h}},\Ffield)=\tilde{\operatorname{H}}_{i-3}(\Delta^v_{\mathbf{h}},\Ffield)=0$.

\begin{teo}
If $\operatorname{max}\mathbf{m}<d-1$, then $\beta_{i-2,id}(P_{d, i})= 0$.
\end{teo}
\begin{proof}
As usual we translate the Betti number computation to our combinatorial setting, 
\[
  \beta_{i-2,id}(P_{d, i})=\sum_{|\mathbf{h}|=id}\operatorname{dim}_\Ffield \tilde{\operatorname{H}}_{i-3}(\Delta^{(i)}_{\mathbf{h}},\Ffield).
\]
and we assume by contradiction that $\tilde{\operatorname{H}}_{i-3}(\Delta^{(i)}_{\mathbf{h}})\neq 0$ and call $w$ a representative of a non trivial homology class.

Because of Proposition \ref{pro-decomposition}, for any $k$ we have the short exact sequence
\[
  0 \rightarrow C_{k}(\Delta^{(i)}_\mathbf{h}\cap \operatorname{L}^i_\mathbf{h}) \stackrel{a_*}{\rightarrow} C_{k}(\Delta^{(i)}_\mathbf{h}) \oplus C_{k}(\operatorname{L}^i_\mathbf{h}) \stackrel{b_*}{\rightarrow} C_{k}(\Delta^v_\mathbf{h})\rightarrow 0,
\]
and so we know that 
\begin{equation}\label{eq-chain-congruence}
C_{k}(\Delta^v_\mathbf{h})\cong \frac{C_{k}(\Delta^{(i)}_\mathbf{h}) \oplus C_{k}(\operatorname{L}^i_\mathbf{h})}{C_{k}(\Delta^{(i)}_\mathbf{h}\cap \operatorname{L}^i_\mathbf{h}) }.
 \end{equation}
It is crucial to observe that $\tilde{\operatorname{H}}_{i-3}(\Delta^v_{\mathbf{h}})=0$, hence $\operatorname{Ker} \partial_{i-3}=Im \partial_{i-2}$, where $\partial$ is the differential map on $\Delta^v_{\mathbf{h}}$.
We also denote by $\partial'$ the direct sum of the differential maps on the simplicial complexes $\Delta^{(i)}_\mathbf{h}$ and $\operatorname{L}^i_\mathbf{h}$. 

Because of (\ref{eq-chain-congruence}), $\partial$ is also a differential map for $\nicefrac{C_{*}(\Delta^{(i)}_\mathbf{h}) \oplus C_{*}(\operatorname{L}^i_\mathbf{h})}{C_{*}(\Delta^{(i)}_\mathbf{h}\cap \operatorname{L}^i_\mathbf{h})}$.
Again from (\ref{eq-chain-congruence}), if $w$ is a cycle in $\operatorname{C}_{i-3}(\Delta^{(i)}_{\mathbf{h}})$, then it is a cycle in $\operatorname{C}_{i-3}(\Delta^{v}_{\mathbf{h}})$.
Since $\tilde{\operatorname{H}}_{i-3}(\Delta^v_{\mathbf{h}})=0$, then there exists $y$ in 
$C_{i-3}(\Delta^{(i)}_\mathbf{h}\cap \operatorname{L}^i_\mathbf{h})$ such that $a_{i-3}(y)=0$. Such $y$ is also a cycle in $C_{i-3}(\Delta^{(i)}_\mathbf{h}\cap \operatorname{L}^i_\mathbf{h})$.

Similarly, $b_{i-3}(w)\in C_{i-3}(\Delta^v_\mathbf{h})$ belongs to $\operatorname{Im} \partial_{i-2}$ because of the vanishing homology of $\Delta^v_{\mathbf{h}}$ in degree $i-3$: there exists $x$ in $C_{i-2}(\Delta^v_\mathbf{h})$, so that $\partial_{i-2} x=b_{i-3}(w)$.

Because of Proposition \ref{pro-dim-intersection}, $C_{i-2}(\Delta^{(i)}_\mathbf{h}\cap \operatorname{L}^i_\mathbf{h})=0$ and  
$C_{i-2}(\Delta^v_\mathbf{h})\simeq C_{i-2}(\Delta^{(i)}_\mathbf{h}) \oplus C_{i-2}(\operatorname{L}^i_\mathbf{h})$.
Hence there exists $z \in C_{k}(\Delta^{(i)}_\mathbf{h}) \oplus C_{k}(\operatorname{L}^i_\mathbf{h})$ with $b_{i-2}(z)=x$.
By commuting of the maps $b_*$ and $\partial_*$, one gets that $\partial_{i-2}(z) =\partial'_{i-2}(z)=w$ and this is against the fact that $w$ is a representative for a non trivial homology class.
\end{proof}

As an immediate corollary, one gets
\begin{cor}\label{theo-linearity}
If $\operatorname{max}\mathbf{m}<d-1$, then $P_{d, i}$ is $(i-2)$--linear.
\end{cor}

Also in this case we provide a close formula for some of the Betti numbers of this ring.

\begin{teo}\label{theo-Betti-max<d-1}
If $\operatorname{max}\mathbf{m}<d-1$, then the Betti table of $P_{d, i}$ is the following:
\[
\begin{array}{c|cccccccccc}
    &0  &1 &2 &\dots &i-2 &i-1 &\dots & d-3& d-2& d-1\\\hline
    0 &1 &0&0 &\dots & 0 & 0 &\dots&0&0&0\\
 
    1 &0 & \beta_{1,2d}& \beta_{2,3d}&\dots& \beta_{i-2,(i-1)d} & \beta_{i-1,id}& \dots & *& 0& 0\\
    2 &0 & 0& 0&\dots& 0 & *& \dots & \beta_{d-3,(d-1)d}& \beta_{d-2,d^2}& 1
      \end{array} 
 \]
where $\beta_{j,(j+1)d}=(d-1)\binom{d-2}{j}-\binom{d}{j}-\binom{d-2}{j+1}$ for all $j=1,\dots, i-1$, $\beta_{d-2,d^2}=\binom{d}{3} -\binom{d}{2}+1$ and $\beta_{d-3,(d-1)d}=\binom{d}{2}-1$.

\end{teo}
\begin{proof}
  Combine the information of Proposition \ref{pro-ultimo-zero}, \ref{pro-necessaria} and Theorem \ref{theo-linearity} together with the knowledge of the h-polynomial in equation (\ref{eq-h-poly-n=2-max<d-1}).
\end{proof}

\section{Canonical modules of the Veronese Modules}

In \cite{noi}, we have studied the Betti numbers of the Veronese modules by means of the reduced homology of the squarefree divisor complex. Moreover, we showed that the Veronese module $S_{n,d,k}$ is Cohen-Macaulay if and only if $k<d$.
In \cite{GotoWatanabe}, Goto and Watanabe showed that the canonical module of the Veronese subring $S_{n,d,0}$ is the Veronese module $S_{n,d,d-n}$. 
In this section we are going to generalize this result for the Veronese module $S_{n,d,k}$.

\vspace{0.1cm}

Usually canonical modules are defined for rings, but in \cite[Section 1.2]{schenzel} Schenzel generalized the definition to finitely generated modules over quotients of local Gorenstein rings. Using the definition of *local ring, we can extend this to our setting.
%
\begin{defi}\cite[Definition 1.5.13]{BrunsHerzog}
A graded ring $A$ is a \emph{*local ring} if it has a unique *maximal ideal, that is, a graded ideal $\mathfrak{m}$ which is not properly contained in any graded ideal $\neq A$.
\end{defi}
\noindent
As done in \cite[Section 3.6]{BrunsHerzog}, it is possible to define the canonical module of a Cohen-Macaulay *local ring. Moreover, similarly to the case of local rings, if $(A, \mathfrak{m})$ is a Cohen-Macaulay *local ring, then it is Gorenstein if and only if the canonical module $\omega_A\cong A(s)$, for some $s\in \inte$.

Now let $M$ be a finitely generated module over a Gorenstein *local ring $(A,\mathfrak{m})$. Consider the minimal free resolution of $M$.
$$
\mathbb{F}.: 0 \rightarrow F_p \rightarrow F_{p-1} \rightarrow \cdots \rightarrow F_0 \rightarrow   0
$$
Let us consider the dual complex $\mathbb{G}.=\mathrm{Hom}_A(\mathbb{F}.,A)$:
$$
\mathbb{G}.: 0 \rightarrow F_0^\vee \rightarrow F_{1}^\vee \rightarrow \cdots \rightarrow F_p^\vee \rightarrow   0
$$
this is exact everywhere because of the Cohen-Macaulay property except for the last map. The cokernel of this last map is by definition $\mathrm{Ext}^p_A(M,A)$. This is the canonical module of $M$ and it will be denoted by $\omega_M$.
\noindent

\begin{TheoM}
Let $k<d$, the canonical module of $S_{n,d,k}$ is $S_{n,d,t}$, with  $t\equiv -n-k \ \mathrm{mod}(d), \ 0\leq t < d$.
\end{TheoM}
\begin{proof}
Consider the decomposition of the polynomial ring in $n$ variables  as a module over $S^{(d)}$, and hence over $R$,
$$S=S_{n,d,0}\oplus S_{n,d,1} \oplus \cdots \oplus S_{n,d,d-1}.$$
Let us consider the quotient ring $S/(\mathfrak{r})$, where $\mathfrak{r}=\{ x_1^d, \dots, x_n^d \}$ is a regular sequence. This ring is Gorenstein and artinian. 
This implies that
$$
\omega_{S/(\mathfrak{r})}\cong \mathrm{Hom}_\field(S/(\mathfrak{r}),\field)\cong S/(\mathfrak{r})(-n(d-1)).
$$
Consider $S_{n,d,k}/(\mathfrak{r})S_{n,d,k}$  and $S/(\mathit{\mathfrak{r}})$ as modules over $S_{n,d,0}/(\mathfrak{r})$, which is isomorphic to $ R/[\mathrm{ker}(\phi)+(\mathfrak{r})]$.
The dual of $S/(\mathfrak{r})$ seen as an $S_{n,d,0}/(\mathfrak{r})$-module is $\mathrm{Hom}_{S_{n,d,0}/(\mathfrak{r})}(S/(\mathfrak{r}), \omega_{S_{n,d,0}/(\mathfrak{r})})$,
and this is isomorphic to 
$$\mathrm{Hom}_\field(S/(\mathfrak{r}),\field)\cong S/(\mathfrak{r})(-n(d-1)).$$
Since the functor $\mathrm{Hom}$ commutes with the direct sum, the dual of $S_{n,d,k}/(\mathfrak{r})$ must be another $S_{n,d,t}/(\mathfrak{r})$, with $k,t<d$. If one considers the module $S_{n,d,k}/(\mathfrak{r})$, when $d\geq n+k$, the maximal graded component to be non zero is $k+(n-1)d$; by dualizing this corresponds to $n(d-1)-k-(n-1)d=d-n-k$; thus the dual of $S_{n,d,k}/(\mathfrak{r})$ is $S_{n,d,d-n-k}/(\mathfrak{r})$. 
In general the dual of $S_{n,d,k}/(\mathfrak{r})$ is $S_{n,d,t}/(\mathfrak{r})$, with $t\equiv -n-k \ \mathrm{mod}(d), \ 0\leq t < d$.
 This is enough to conclude that the canonical module of $S_{n,d,k}$ is $S_{n,d,t}$ because $\mathfrak{r}$ is a regular sequence of length as the dimension of the module (see \cite[Chapter 21]{eisenbudbook}).
\end{proof}
\noindent
One can also obtain the conclusion about the integer $t$, using \cite[Theorem 1]{Spaul}: with the duality formula one can check that the Betti table of $S_{n,d,k}$ is the Betti table of $S_{n,d,d-n-k}$ reversed.\\

\addcontentsline{toc}{section}{Bibliography}
\bibliographystyle{siam}
\bibliography{Veronese}

\vspace{0.5cm}

 {\scshape Ornella Greco}, \texttt{ogreco@kth.se}\\
 {\scshape Ivan Martino}, \texttt{i.martino@northeastern.edu}
\end{document}